\documentclass[11pt, a4paper]{amsart}

\usepackage{amsfonts,amsmath,amssymb, amscd,fullpage}
\usepackage[all]{xy}

\newtheorem{theorem}{Theorem}[section]
\newtheorem{lemma}[theorem]{Lemma}
\newtheorem{definition}[theorem]{Definition}
\newtheorem{proposition}[theorem]{Proposition}
\newtheorem{corollary}[theorem]{Corollary}
\newtheorem{question}[theorem]{Question}

\theoremstyle{remark}
\newtheorem{remark}[theorem]{Remark}

\theoremstyle{definition}

\newcommand\pf{\begin{proof}}
\newcommand\epf{\end{proof}}

\newcommand{\Z}{{\mathbb Z}}

\newcommand{\C}{{\mathbb C}}

\newcommand\id{\operatorname{id}}

\newcommand\ad{\operatorname{ad}}

\newcommand\co{\operatorname{co}}

\newcommand\End{\operatorname{End}}

\newcommand\GL{\operatorname{GL}}
\newcommand\SU{\operatorname{SU}}
\newcommand\SO{\operatorname{SO}}

\DeclareMathOperator{\Ker}{Ker}

\numberwithin{equation}{section}

\title{Examples of inner linear Hopf algebras}

\author{Nicol\'as Andruskiewitsch}
\address{Nicol\'as Andruskiewitsch:
Facultad de Matem\'atica, Astronom\'ia y F\'isica, Universidad National de C\'ordoba.
CIEM - CONICET. Medina Allende s/n (5000) Ciudad Universitaria, C\'ordoba, Argentina}
\email{andrus@famaf.unc.edu.ar}

\author{Julien Bichon}
\address{Julien Bichon:
Clermont Universit\'e, Universit\'e Blaise Pascal, Laboratoire de Math\'ematiques CNRS UMR 6620,
63177 Aubi\`ere, France }
\email{Julien.Bichon@math.univ-bpclermont.fr}

\thanks{Work supported by the ANR project ''Galoisint'' BLAN07-3\_183390}

\begin{document}

\maketitle

\begin{abstract}
The notion of inner linear Hopf algebra is a generalization
of the notion of discrete linear group. In this paper,
we prove two general results that enable us to enlarge the class of Hopf algebras
that are known to be inner linear:
the first one is a characterization by using the Hopf dual, while
the second one is a stability result under extensions.
We also discuss the related notion of inner unitary Hopf $*$-algebra.
\end{abstract}

\section{introduction}
Throughout the paper, we work over $\C$, the field of complex numbers.

The notion of inner linear Hopf algebra was introduced in \cite{bb}
as a natural generalization of the notion of discrete linear group.
The precise definition is the following one.

\begin{definition}
A Hopf algebra is said to be inner linear if it contains an
ideal of finite codimension that does not contain any non-zero Hopf ideal.
\end{definition}

Indeed, when $H= \C[\Gamma]$ is the group algebra of a discrete group $\Gamma$, then
$H$ is inner linear if and only if the group $\Gamma$ is linear in the usual sense, i.e. admits
a faithful finite-dimensional linear representation \cite{bb}.

Also the concept of inner linearity for Hopf algebras generalizes the
notion of linear (= finite dimensional) Lie algebra: if $H=U(\mathfrak g)$
is the enveloping algebra of a Lie algebra $\mathfrak g$, then $U(\mathfrak{g})$
is inner linear if and only if $\mathfrak g$ is finite-dimensional.

We believe that the problem to know whether a given Hopf algebra
is inner linear or not is an important one, since it is a generalization of
the celebrated linearity problem for discrete groups.

Several examples were considered in \cite{bb}, and we continue
this study here.
We prove two general results that enable us to enlarge the class of Hopf algebras
that are known to be inner linear.

\begin{enumerate}
 \item We give a reformulation of inner linearity using the Hopf dual.
 This enables us to show that
 the Drinfeld-Jimbo quantum algebras attached to semisimple Lie algebras are inner linear
 if the parameter $q$ is not a root of unity.

\item We prove a stability result for inner linearity under extensions.
This applies to Drinfeld-Jimbo algebras and quantized function algebras at roots of unity, and
to the recently introduced half-liberated orthogonal Hopf algebras \cite{bs}.
\end{enumerate}

The paper is organized as follows.
In Section 2 we reformulate the notion of inner linear Hopf algebra by using
the Hopf dual,  with, as an application,
the inner linearity of the Drinfeld-Jimbo quantum algebras $U_q(\mathfrak g)$
and $\mathcal O_q(G)$ if $q$ is not a root of unity.
 Section 3 contains some basic results
on the possible use of quotient Hopf algebras to show inner linearity, which
 might be used to show the inner linearity
of Hopf algebras having an analogue of the dense big cell of reductive
algebraic groups, such as in \cite{pw}.
 In Section 4 we give a
 stability result for inner linearity under extensions.
 This applies to quantum algebras $\mathcal O_q(G)$ and $U_q(\mathfrak g)$ at roots of unity,
 as well as to the half-liberated Hopf algebras $A_o^*(n)$ from \cite{bs}.
In Section 5 we study the related notion
 of inner unitary Hopf $*$-algebra, and it is shown
 that for $q \in \mathbb R^*$, $q\not=\pm 1$, and $K$ a connected simply connected
 simple compact Lie group, the Hopf $*$-algebra $\mathcal O_q(K)$
 is not inner unitary (while it is inner linear as a Hopf algebra).
 We also give a Hopf $*$-algebra version of the extension theorem of Section 4.

We assume that the reader is familiar with the basic notions of Hopf algebras,
for which \cite{mo} is a convenient reference.
Our terminology and notation are the standard ones: in particular, for a Hopf algebra, $\Delta$, $\varepsilon$ and $S$ denote the comultiplication, counit and antipode, respectively.

\section{Inner linear Hopf algebras and the Hopf dual}

In this section we reformulate the notion of linear Hopf algebra by using
the Hopf dual, and we apply this reformulation to Drinfeld-Jimbo quantum algebras.

\begin{theorem}\label{reformulation}
A Hopf algebra $H$ is inner linear if and only if $H^0$, the Hopf dual of $H$,
contains a finitely generated Hopf subalgebra that separates the points of $H$.
\end{theorem}

Of course finitely generated Hopf algebra means
finitely generated as a \textsl{Hopf algebra}.
Before proving Theorem \ref{reformulation},
we need to recall the concept of inner faithful representation  of a Hopf algebra \cite{bb},
which generalizes the notion of faithful representation of a discrete group.

\begin{definition}
Let $H$ be a Hopf algebra and let $A$ be an algebra. A representation $\pi : H \longrightarrow A$ is said to be inner faithful if
$\Ker(\pi)$ does not contain any non-zero Hopf ideal.
\end{definition}

It is clear that a Hopf algebra is inner linear if and only
it admits a finite-dimensional  inner faithful representation.

\begin{proof}[Proof of theorem \ref{reformulation}]
We first introduce some notation.
Let $F$ be the free monoid generated by the set $\mathbb N$, with its
generators denoted $\alpha_0, \alpha_1 ,\ldots$ and its unit element denoted $1$.
Let $V$ be a vector space.
To any element $g \in F$, we associate a vector space $V^g$, defined inductively on the length
of $g$ as follows.
We put $V^1=k$, $V^{\alpha_k} = V^{*\cdots * (k \ {\rm times})}$ (so that $V^{\alpha_0} = V, V^{\alpha_1} = V^*$...). Now for $g,h \in F$ with $l(g) > 1$ and $l(h)>1$, we put
$V^{gh}= V^g \otimes V^h$. For $g \in F$, we have, if $V$ is finite dimensional, canonical isomorphisms
$\End(V^g) \cong \End(V)^g$, where the algebra ${\rm End}(V)^g$ is defined as in
Section 2 of \cite{bb}.

Now let $\pi : H \longrightarrow \End(V)$ be a representation, with $V$ finite dimensional, so that
$V$ is an $H$-module. For any $g \in F$, the standard procedure gives
an $H$-module structure on $V^g$, and we identify, via the algebra isomorphism
$\End(V^g) \cong \End(V)^g$, the corresponding algebra map
$H \longrightarrow \End(V^g)$ with $\pi^g : H \longrightarrow \End(V)^g$
as defined in Section 2 of \cite{bb}.

We know, from Proposition 2.2 in \cite{bb},
that the largest Hopf ideal contained in $\Ker(\pi)$ is $I_\pi= \cap_{g \in F} \Ker(\pi^g)$.
Hence $\pi$ is inner faithful if and only $\cap_{g \in F} \Ker(\pi^g)=(0)$.

 Thus if $\pi$ is inner faithful, we have $\cap_{g \in F} \Ker(\pi^g)=(0)$. The coefficients
of the representation $\pi^g$ belong, by construction, to $L$, the Hopf
subalgebra of $H^0$ generated by the coefficients of $\pi = \pi^{\alpha_0}$. Hence $L$
separates the points of $H$.

The proof of the first implication in the theorem follows.
If $H$ is inner linear, let $\pi : H \longrightarrow A$ be
an inner faithful representation, with $A$ finite-dimensional. We can assume, by using the regular representation of $A$, that $A={\rm End}(V)$ for some finite-dimensional
vector space $V$. Hence the finitely generated Hopf subalgebra $L \subset H^0$
constructed above separates the points of $H$, by the previous discussion.

Conversely, assume that we have a  finitely generated Hopf subalgebra $L \subset H^0$
that separates the points of $H$. Then $L$ is generated by the coefficients
of a finite dimensional $H^0$-comodule, corresponding to a finite dimensional
$H$-module $V$. Let $\pi : H \longrightarrow \End(V)$ be the corresponding algebra map.
The elements of $L$ are the coefficients of the $H$-modules $V^g$ constructed above,
hence $\cap_{g \in F} \Ker(\pi^g) = (0)$ since $L$ separates the points of $H$, and we conclude
that $\pi$ is inner faithful, so that $H$ is inner linear.
\end{proof}

\begin{corollary}\label{corosep}
Let $H$ be a Hopf algebra such that $H^0$ separates the points of $H$.
\begin{enumerate}
\item Assume that $H^0$ is finitely generated as a Hopf algebra. Then
$H$ is inner linear.
\item Assume that $H$ is finitely generated as a Hopf algebra. Then
$H^0$ is inner linear.
\end{enumerate}
\end{corollary}

\begin{proof}
The first asssertion follows from the previous theorem.
For the second one, consider the embedding
$H \subset (H^0)^0$ (this is indeed an embedding since $H^0$
separates the points of $H$). Then $H$ is a finitely generated Hopf
subalgebra of $(H^0)^0$, and we conclude by the previous theorem.
\end{proof}

As an application, we get the following result for quantum groups at generic $q$; the case
when $q$ is a root of 1 will be discussed in Section 4.

\begin{theorem}
Let $G$ be a (complex) connected, simply connected, semisimple algebraic
group with Lie algebra $\mathfrak g$, and
let $q\in \C^*$. If $q$ is not a root of unity, then
the Hopf algebras $U_q(\mathfrak{g})$ and $\mathcal O_q(G)$
are inner linear.
\end{theorem}

\begin{proof}
We know from the representation theory of $U_q(\mathfrak{g})$
(see \cite[Lemma 8.3]{jole})
that the type I irreducible representations of $U_q(\mathfrak g)$
separate the points of $U_q(\mathfrak{g})$.
Hence the linear span of their coefficients separates
the elements of $U_q(\mathfrak g)$, and  is a finitely generated
Hopf subalgebra of $U_q(\mathfrak{g})^0$ (actually it is $\mathcal O_q(G)$,
see \cite{jan2,bg}). It follows from Theorem \ref{reformulation}
that $U_q(\mathfrak g)$ is inner linear.
The second part of the previous corollary  ensures that $U_q(\mathfrak{g})^0$ is inner linear, and
hence so is the Hopf subalgebra   $\mathcal O_q(G)$.
\end{proof}

\begin{remark}
 The proof uses the fact that $U_q(\mathfrak{g})^0$ separates
 the points of $U_q(\mathfrak{g})$,  a deep result in the  representation
theory of $U_q(\mathfrak{g})$. On the other hand, to prove this separation result,
it might be simpler to combine Theorem \ref{reformulation}
and  the inner faithfulness criterion for representations
of pointed Hopf algebras in \cite{bb} (Theorem 4.1).
This last theorem is as follows: a pointed Hopf algebra $H$
is inner linear if and only if there exists a finite-dimensional
representation $\pi : H \longrightarrow A$ such that for any group-like $g \in {\rm Gr}(H)$, the restriction map
$\pi_{|\mathcal P_{g,1}(H)} : \mathcal P_{g,1}(H) \longrightarrow A$ is injective.
A similar  idea to prove separation results, using description of skew-primitives, was already used in \cite{cm}, Section 4.
\end{remark}


\section{Inner linearity and quotient Hopf algebras}

In this section we prove some very basic but useful results
on the possible use of quotient Hopf algebras to show inner linearity.
We begin with a  lemma.

\begin{lemma}
 Let $H$ be a Hopf algebra, let $I_1,I_2$ be Hopf ideals in $H$.
 Let $\rho_k : H/I_k \longrightarrow A$, $k=1,2$,
 be some representations. Consider the representation
 \begin{align*}
\rho  : H &\longrightarrow A \times B \\
x &\longmapsto (\rho_1\circ \pi_1(x) , \rho_2\circ \pi_2(x))
 \end{align*}
 where $\pi_k$, $k=1,2$ is the canonical projection.
 Assume that $I_1 \cap I_2$ does not contain any
 non zero Hopf ideal and that $\rho_1$ and $\rho_2$ are inner faithful.
 Then $\rho$ is inner faithful.
\end{lemma}

\begin{proof}
 Let $J$ be a Hopf ideal contained in $\Ker(\rho) = \Ker(\rho_1 \circ \pi_1) \cap \Ker(\rho_2 \circ \pi_2)$. Then for $k=1,2$, $\pi_k(J)$  is a Hopf ideal contained in $\Ker(\rho_k)$, and
 hence $J \subset I_k$ by inner faithfulness of $\rho_k$. Hence
 $J$ is a Hopf ideal in $I_1 \cap I_2$, and $J=(0)$, which proves that $\rho$ is inner faithful.
\end{proof}

\begin{corollary}
 Let $H$ be a Hopf algebra, let $I_1,I_2$ be Hopf ideals in $H$. Assume
 that $I_1 \cap I_2$ does not contain any non zero Hopf ideal and that the Hopf algebras
 $H/I_1$ and $H/I_2$ are inner linear. Then $H$ is inner linear.
\end{corollary}

\begin{proof}
 This is a consequence of the lemma, by using finite-dimensional
 inner faithful representations of $H/I_1$ and $H/I_2$.
\end{proof}

\begin{corollary}
 Let $H$ be a Hopf algebra, let $I_1,I_2$ be Hopf ideals in $H$. Assume
 that the algebra map
 \begin{align*}
  \theta : H &\longrightarrow H/I_1 \otimes H/I_2 \\
  x &\longmapsto \pi_1(x_{(1)}) \otimes \pi_2(x_{(2)})
 \end{align*}
where $\pi_1,\pi_2$ are the canonical projections, is injective, and that the Hopf algebras
 $H/I_1$ and $H/I_2$ are inner linear. Then $H$ is inner linear.
\end{corollary}

\begin{proof}
 Let $J \subset I_1 \cap I_2$ be a Hopf ideal.
 Then $$\Delta(J) \subset J \otimes H + H \otimes J \subset
 (I_1 \cap I_2) \otimes H + H \otimes (I_1 \cap I_2).$$
 Hence $\theta(J)=(0)$ and $J=(0)$ by the injectivity of $\theta$, and we are done by the previous result.
\end{proof}

This last result might be used to show the inner linearity
of Hopf algebras having an analogue of the dense big cell of reductive
algebraic groups, such as in \cite{pw}.
Indeed for $\mathcal O_q(\GL_n(\C))$, the previous result
combined with Theorem 8.1.1 in \cite{pw}
reduces the problem to show the inner linearity of
$\mathcal O_q(\GL_n( \C))$ to show the inner linearity of the
pointed Hopf algebras $\mathcal O_q(B)$ and $\mathcal O_q(B')$, for which the method
of Theorem 4.1 in \cite{bb} is available.

\section{Stability of inner linearity under extensions}

We now study the question of the stability of inner linearity under extensions.
At the group level, it is known that linearity is not stable under extensions (see e.g. \cite{fp}),
but we have the following positive result: If $G$ is a group having a linear normal subgroup $H$
of finite index, then $G$ is linear.
We prove a Hopf algebraic analogue
of a weak form of this result, and we apply it to two types of Hopf algebras.

\subsection{The general result}
Here is our more general result on the preservation of inner linearity by extensions.

\begin{theorem}\label{extension}
 Let $H$ be a Hopf algebra and let $A\subset H$ be a normal Hopf subalgebra.
 Assume that the following conditions hold:
 \begin{enumerate}
  \item
 $A$ is inner linear and commutative,
 \item  $H$ is finitely generated as a right $A$-module.
  \end{enumerate}
 Then $H$ is inner linear.
\end{theorem}


To prove the theorem, we need to recall some facts
on exact sequences of Hopf algebras.

First recall that
a Hopf subalgebra $A$ of a Hopf algebra $H$ is said to be \emph{normal} if it is
stable under both left and right adjoint actions of $H$ on itself,
defined, respectively, by
$$\ad_l(x)(y) = x_{(1)}y\mathcal S(x_{(2)}), \quad \ad_r(x)(y) = \mathcal S(x_{(1)})y
x_{(2)},$$ for all $x,y \in H$. If $A \subseteq H$ is a normal Hopf
subalgebra, then the ideal $HA^+=A^+H$ is a Hopf ideal and the
canonical map $H \to \overline H : = H/HA^+:=H//A$ is a Hopf algebra
map.

Now recall that a sequence
of Hopf algebra maps
\begin{equation}\label{seq}k \to A \overset{i}\to H \overset{p}\to \overline H \to
k,\end{equation} is called \emph{exact} if the following
conditions hold:
\begin{enumerate}\item $i$ is injective and $p$ is surjective,
\item $p \circ i$ = $\epsilon 1$,
\item ${\rm Ker} p =HA^+$, \item $A = H^{\co p} = \{ h \in H:\, (\id \otimes p)\Delta(h) = h \otimes 1
\}$. \end{enumerate}
It follows that $i(A)$ is normal Hopf subalgebra of $H$.
Conversely, if
we have a sequence \eqref{seq} and  $H$
is faithfully flat over $A$, then (1), (2) and (3) imply (4). See e.g. \cite{ad}.

\medskip

We now state several preparatory lemmas for the proof of Theorem \ref{extension}.

\begin{lemma}\label{exactcommutative}
 Let $H$ be a Hopf algebra and let $A\subset H$ be a normal Hopf subalgebra.
 If $A$ is commutative, then $k \to A \to H \to H//A \to k$ is an exact sequence.
\end{lemma}

\begin{proof}
 We know from \cite{ag}, Proposition 3.12, that $H$ is a faithfully flat
 as an $A$-module, and hence we have the announced exact sequence
 by the previous considerations.
\end{proof}

We shall need the following Hopf algebraic version of the 5 lemma.

\begin{lemma}\label{5lemma}
 Consider a commutative diagram of Hopf algebras
 \begin{equation}\begin{CD}
  k @>>> A @>>> H @>>> \overline{H}@>>> k \\
  @. @| @VV{\theta}V  @| @. \\
    k @>>> A @>>> H' @>>> \overline{H} @>>> k
 \end{CD}\end{equation}
 where the rows are exact. If $A$ is commutative, then $\theta$ is an isomorphism.
\end{lemma}

\begin{proof}
 The proof is similar to Corollary 1.15 in \cite{ag1}: $A \subset H$ and
 $A \subset H'$ are $\overline{H}$-Galois extensions and since $H'$ is a faithfully flat $A$-module
 \cite[Proposition 3.12]{ag}, the $\overline{H}$-colinear $A$-linear algebra map $\theta$ is an isomorphism
 by Remark 3.11 in \cite{schn}.
\end{proof}

\begin{lemma}
 Let $A \subset H$ be a normal and commutative Hopf subalgebra.
 Let $J$ be a Hopf ideal in $H$ such that $J \cap A=(0)$ and $J \subset A^+H$. Then $J=(0)$.
\end{lemma}

\begin{proof}
 We have, by Lemma \ref{exactcommutative}, an exact sequence
 \begin{equation*}k \to A \overset{i}\to H \overset{p_H}\to H//A \to
k.\end{equation*}
Now put $K = H/J$, and let $q : H \to K$ be the canonical surjection.
 Since $J\cap A= (0)$, we have an injective
Hopf algebra map $j : A \to K$ such that $q \circ i= j$, and $j(A)$ is a normal Hopf subalgebra
of $K$. We then have an exact sequence
\begin{equation*}k \to A \overset{j}\to K \overset{p_K}\to K//A \to
k.\end{equation*}
We now claim that there exists a Hopf algebra isomorphism $\overline{q} : H \to K$
such that the following diagram is commutative
\begin{equation*}\begin{CD}
  k @>>> A @>>> H @>{p_H}>> H//A @>>> k \\
  @. @| @VV{q}V  @VV{\overline{q}}V @. \\
    k @>>> A @>>> K @>{p_K}>> K//A @>>> k
 \end{CD}\end{equation*}
 Indeed, we have $p_K \circ q(A^+H)= p_K(j(A)^+H)=0$, which shows the existence
 of $\overline{q}$. We also have, since $J \subset A^+H$,
 \begin{align*}
 {\rm Ker}(\overline{q}) &= p_H({\rm Ker}(p_K\circ q)) = p_H(q^{-1}({\rm Ker}(p_K)) \\
 &=p_H(q^{-1}(j(A)^+K))= p_H(J+A^+H) \subset p_H(A^+H) =(0)
 \end{align*}
 and hence $\overline{q}$ is an ismorphism.
 We get  a commutative diagram with exact rows
 \begin{equation*}\begin{CD}
  k @>>> A @>>> H @>{\overline{q} \circ p_H}>> K//A @>>> k \\
  @. @| @VV{q}V  @| @. \\
    k @>>> A @>>> K @>>> K//A @>>> k
 \end{CD}\end{equation*}
(the top row is still exact because $\overline{q}$ is an isomorphism) and Lemma \ref{5lemma} ensures that $q$ is an isomorphism, hence $J=(0)$.
\end{proof}

\begin{proof}[Proof of Theorem \ref{extension}]
Let $\rho : A \to \End(V)$ be an inner faithful representation, with $V$ finite-dimensional.
As usual, we get the induced representation
$\tilde{\rho} : H \to \End(H\otimes_AV)$. Since $H$ is finitely generated as an $A$-module,
the vector space $H\otimes_AV$ and the Hopf algebra
$H//A$ are finite dimensional.  We thus consider the finite-dimensional representation
\begin{align*}
 \theta : H &\longrightarrow H//A \times \End(H\otimes_AV) \\
 x &\longmapsto (p_H(x), \tilde{\rho}(x))
\end{align*}
Let us show that $\theta$ is inner faithful. Let $J \subset {\rm Ker}(\theta)= A^+H \cap {\rm Ker}(\tilde{\rho})$
be a Hopf ideal. Then $J\cap A \subset A$ is a Hopf ideal. It is easy
to see, using the faithful flatness of $H$ as a right $A$-module,
that ${\rm Ker}(\tilde{\rho}) \cap A \subset {\rm Ker}(\rho)$. Thus
  $J\cap A \subset {\rm Ker}(\rho)$ and $J \cap A=(0)$ since $\rho$ is inner faithful.
  We thus have $J \subset A^+H$ and $J \cap A=(0)$: the previous lemma ensures that
  $J=(0)$. Hence $\theta$ is inner faithful and $H$ is inner linear.
\end{proof}

\begin{question}
 Is true that the induced representation $\tilde{\rho} : H \to \End(H\otimes_AV)$
 is inner faithul if $\rho$ is? A positive answer would give a strenghtening
 of Theorem \ref{extension}, dropping the commutativity assumption on $A$.
\end{question}

 \subsection{Applications}
 Our first application is with quantized function algebras at roots of unity.

 \begin{theorem}
 Let $G$ be a connected, simply connected complex semisimple algebraic
 group with Lie algebra $\mathfrak g$, and let $q$ be a root of unity of odd order $\ell$, with $\ell$ prime to $3$ if $G$
 contains a component of type $G_2$. The Hopf algebras $\mathcal O_q(G)$ and $U_q(\mathfrak g)$
 are inner linear.
\end{theorem}

\begin{proof}
 It is known \cite{dcl} that $\mathcal O_q(G)$ contains a central Hopf subalgebra
 isomorphic to $\mathcal O(G)$, and that $\mathcal O_q(G)$
 is finitely generated and projective as $\mathcal O(G)$-module.
 The Hopf algebra $\mathcal O(G)$ is inner linear \cite{bb}, and thus
 Theorem \ref{extension} gives the result.

 Similarly it is known (see e.g. \cite{bg}) that $U_q(\mathfrak g)$ contains
 an affine central Hopf subalgebra (hence inner linear by \cite{bb})
 $Z_0$  such that $U_q(\mathfrak g)$ is a finitely generated $Z_0$-module.
 Hence again Theorem \ref{extension} gives the result.
\end{proof}

We now turn to the half-liberated orthogonal Hopf algebra.
Recall that the half-liberated orthogonal Hopf algebra
$A_o^*(n)$, introduced in \cite{bs} and further studied in \cite{bv}, is the algebra presented by generators $u_{ij}$, $1 \leq i,j\leq n$,
submitted to the relations
\begin{enumerate}
 \item the matrix $u =(u_{ij})$ is orthogonal,
\item $u_{ij}u_{kl}u_{pq} = u_{pq} u_{kl} u_{ij}$, $1 \leq i,j,k,l,p,q \leq n$.
\end{enumerate}
It admits a Hopf algebra structure given by the standard formulas
$$\Delta(u_{ij}) = \sum_k u_{ik} \otimes u_{kj}, \quad \varepsilon(u_{ij})=\delta_{ij}, \quad S(u_{ij})=u_{ji}$$

\begin{theorem}
 The Hopf algebra $A_o^*(n)$ is inner linear.
\end{theorem}

\begin{proof}
 Let $A\subset H=A_o^*(n)$ be the subalgebra generated by the elements
 $u_{ij}u_{kl}$. As remarked in \cite{bv}, it is a commutative Hopf subalgebra of $H$.
 Thus $A \simeq \mathcal O(G)$ for a (reductive) algebraic group $G$
 (in fact it is shown in \cite{bv} that $G= {\rm PGL}_n(\C)$)) and $A$ is inner linear by \cite{bb}.
 It is easy to check the existence of a Hopf algebra map
 $\pi: H \to \C[\Z_2]$, $u_{ij} \mapsto \delta_{ij}g$, where $1 \not= g \in \Z_2$, and
 that $A = H^{\co \pi}$. This Hopf algebra map is cocentral ($\pi(x_{(1)}) \otimes x_{(2)}= \pi(x_{(2)}) \otimes x_{(1)}$, for all $x \in H$) and hence $A$ is normal in $H$
 (see e.g. Lemma 3.4.2 in \cite{mo}).
 Moreover $H$ is, as an $A$-module, generated by the elements $u_{ij}$.
 Thus $H$ is inner linear by Theorem \ref{extension}.
\end{proof}

\section{Inner unitary Hopf $*$-algebras}

In this section we discuss the Hopf algebraic
analogue of the notion of discrete unitary group.
Here we say that a discrete group is unitary if
it can be embedded as a subgroup of the group
of unitary operators on a finite-dimensional Hilbert space.
We work in the framework of Hopf $*$-algebras (see e.g. \cite{ks} for the relevant definitions).

\begin{definition}
 A  Hopf $*$-algebra $H$ is said to be \textbf{inner unitary}
if there exists a  $*$-representa\-tion $\pi : H \longrightarrow A$ into
a finite-dimensional $C^*$-algebra $A$ such that ${\rm Ker}(\pi)$ does not contain
any non zero Hopf $*$-ideal.
\end{definition}

Of course a group $\Gamma$ is unitary if and only if the
Hopf $*$-algebra $\C[\Gamma]$ is inner unitary.

A possible trouble with the previous natural definition
is that it is not clear that an inner unitary Hopf algebra
must be inner linear. We do not know if this is true in general, but
we shall see that under some mild assumptions (see Proposition \ref{iuil}),
an inner unitary Hopf $*$-algebra
is inner linear. This is true in particular for compact Hopf algebras,
i.e. Hopf $*$-algebras arising from compact quantum groups, the
class of Hopf $*$-algebras we are most interested in.

We need the following concept.

\begin{definition}
 We say that a Hopf algebra $H$ has a regular antipode
if there exists a group-like $a \in H$, an algebra morphism
$\Phi : H \longrightarrow \C$ and an integer $m \geq 1$ such that
$$\forall x \in H, \ S^{2m}(x) = a(\Phi * {\rm id}_H * \Phi^{-1}(x))a^{-1}$$
\end{definition}

For example, by \cite{bebuto}, any co-Frobenius Hopf algebra has a Radford type formula
for $S^4$ and hence has a regular antipode. In particular
any cosemisimple Hopf algebra has a regular antipode.

\begin{proposition}
 Let $H$ be a Hopf algebra having a regular antipode.
The following assertions are equivalent.
\begin{enumerate}
 \item $H$ is inner linear.
\item There exists a representation $\pi : H \longrightarrow A$ into  a finite-dimensional algebra $A$ such that ${\rm Ker}(\pi)$ does not contain any
non-zero Hopf ideal $I$ such that $S(I)=I$.
\end{enumerate}
\end{proposition}

\begin{proof}
By the assumption there exists a group-like $a\in H$, an algebra morphism
 $\Phi : H \longrightarrow \C$ and $m \geq 1$ such that for $x \in H$, one has
$S^{2m}(x)= a(\Phi * {\rm id}_H* \Phi^{-1}(x))a^{-1}$.
Let $\pi : H \longrightarrow A$ be a representation satisfying condition (2), and
consider the representation
\begin{align*}
\pi' : H &\longrightarrow A \times \C^2 \\
x &\longmapsto (\pi(x), \Phi(x), \Phi^{-1}(x))
\end{align*}
Let $J \subset {\rm Ker}(\pi')$ be a Hopf ideal. Then we have
\begin{align*}
S^{-1}(J) & =  a^{-1}\left(\Phi^{-1}* S^{2m-1}*\Phi)(J)\right)a \\ & \subset
a^{-1}\left(\Phi^{-1}(J)S^{2m-1}(H)\Phi(H) + \Phi^{-1}(H)S^{2m-1}(J)\Phi(H)
+\Phi^{-1}(H)S^{2m-1}(H)\Phi(J) \right)a \\ & \subset a^{-1}S^{2m-1}(J)a \subset J
\end{align*}
since $\Phi(J)= \Phi^{-1}(J)=(0)$. Thus $S(J)=J$ with $J \subset {\rm Ker}(\pi)$, and
hence $J=(0)$. Thus $\pi'$ is inner faithful and $H$ is inner linear.
\end{proof}

\begin{proposition}\label{iuil}
 Let $H$ be a Hopf $*$-algebra having a regular antipode.
If $H$ is inner unitary, then $H$ is inner linear. In particular
an inner unitary compact Hopf algebra is inner linear.
\end{proposition}

\begin{proof}
Let  $\pi : H \longrightarrow A$ be a $*$-representation into
a finite-dimensional $C^*$-algebra $A$ such that ${\rm Ker}(\pi)$ does not contain
any non zero Hopf $*$-ideal. The previous proposition ensures that to show $H$ is inner linear, it is enough to check that ${\rm Ker}(\pi)$ does not contain
any non-zero Hopf ideal $I$ with $S(I)=I$. So let $I \subset {\rm Ker}(\pi)$
be such a Hopf ideal. It is clear that $I+I^*$ is a $*$-bi-ideal contained in
${\rm Ker}(\pi)$. Moreover $S(I+I^*) = S(I)+S(I^*)=S(I) + S^{-1}(I)^* = I+I^*$.
Hence $I+I^*$ is a Hopf $*$-ideal contained in ${\rm Ker}(\pi)$ and $I+I^*=(0)=I$.
\end{proof}

It is well-known that, already at the discrete  group level, the converse of this result is not true.
 For example the group $\Gamma = \langle x, y \ | \ yxy^{-1} =x^2 \rangle$
 is linear but is not unitary.
For non cocommutative Hopf algebras, we also have
the following example.

\begin{proposition}
Let $G$ be a connected, simply connected and simple complex Lie group and let
$K \subset G$ be a maximal compact subgroup.
 Let $q \in \mathbb R^*$, $q \not=\pm1$. Then the compact Hopf algebra $\mathcal O_q(K)$
is inner linear but is not inner unitary.
\end{proposition}

\begin{proof}
 The Hopf algebra underlying $\mathcal O_q(K)$ is $\mathcal O_q(G)$, hence
is inner linear by Theorem 2.4. Assume that $\mathcal O_q(K)$ is inner unitary:
there exists a $*$-algebra morphism $$\pi :\mathcal O_q(K)\longrightarrow \mathcal
B(V)$$
where $V$ is a finite-dimensional Hilbert space, such that  ${\rm Ker}(\pi)$ does not contain
any non zero Hopf $*$-ideal. The $\mathcal O_q(K)$-module $V$ is semisimple
since $\pi$ is a $*$-algebra map. We know from \cite{koso}
that the finite-dimensional irreducible
Hilbert space representations  of $\mathcal O_q(K)$ all are one-dimensional, and
hence the elements of $\pi(\mathcal O_q(K))$ are simultaneously
diagonalizable. It follows that $\pi(\mathcal O_q(K))$ is a commutative $*$-algebra
and that the commutator ideal of $\mathcal O_q(K)$, which is a Hopf $*$-ideal,
is contained in ${\rm Ker}(\pi)$. Hence the commutator ideal is zero
and $\mathcal O_q(K)$ is commutative: a contradiction (these last arguments
are from Proposition 2.13 in \cite{bb}).
\end{proof}

\begin{remark}
 The Hopf $*$-algebra $\mathcal O_{-1}({\rm SU}_2)$ is inner unitary, since
 the inner faithful representation $\mathcal O_{-1}({\rm SU}_2) \rightarrow M_2(\C) \otimes \C^4$
 constructed in \cite{bb}, Corollary 6.6, is a $*$-algebra map. Also
 one can adapt the other constructions in Section 6 of \cite{bb}
 to show that $ \mathcal O_{-1}({\rm SU}_n)$ is inner linear for any $n$.
\end{remark}

It is now natural to ask if the compact Hopf $*$-algebra $A_o^*(n)$
(whose $*$-structure is defined by $u_{ij}^*=u_{ij}$) is inner unitary.
For this we need  a Hopf $*$-algebra version of Theorem \ref{extension}.

\begin{theorem}\label{extensionunit}
 Let $H$ be a compact Hopf algebra and let $A\subset H$ be a normal Hopf $*$-subalgebra.
 Assume that the following conditions hold:
 \begin{enumerate}
  \item
 $A$ is inner unitary and commutative,
 \item  $H$ is finitely generated as a right $A$-module.
  \end{enumerate}
 Then $H$ is inner unitary.
\end{theorem}

\begin{proof}
 The proof is an adaptation of the proof of Theorem \ref{extension}, using induced
 representations of $C^*$-algebras \cite{rie}.
 Similarly to Section 4,
 $H$ is a faithfully flat $A$-module since $A$ is commutative, and hence
 $H^{\co p}=A$, where $p : H \longrightarrow H//A$ is the canonical map. Hence
 using the Haar measure $\phi$ on the compact Hopf algebra $H//A$, we
 get a map
 $$E = ({\rm id}\otimes \phi) \circ ({\rm id} \otimes p) \circ \Delta : H \longrightarrow A$$
 which is a conditional expectation in the sense of \cite{rie} (see e.g. \cite{po}).

 Consider now a Hilbert space $*$-representation $\rho : A \to \mathcal B (V)$.
 There is a sesquilinear form on $H\otimes_A V$ such that for $x,y \in H$, $v, w \in V$,
 we have
 $$\langle x \otimes_A v, y \otimes_A w \rangle = \langle \rho(E(y^*x))(v) ,w \rangle$$
Moreover this is a pre-inner product (Lemma 1.7 in \cite{rie}). Killing the norm zero
elements, we get a (finite-dimensional) Hilbert space $H\underline{\otimes}_A V$, and
an induced $*$-representation (Theorem 1.8 in \cite{rie})
\begin{align*}
 \tilde{\rho} : H & \longrightarrow \mathcal B(H\underline{\otimes}_AV) \\
 x & \longmapsto \tilde{\rho(x)}, \ \tilde{\rho(x)}(y \underline{\otimes}_A v) = xy  \underline{\otimes}_A v
\end{align*}
We wave     ${\rm Ker}(\tilde{\rho}) \cap A \subset {\rm Ker}(\rho)$
since the map $V \rightarrow H \underline{\otimes}_A V$,  $v \mapsto1 \underline{\otimes}_A v$,
is isometric.
 Then, similarly to the proof of Theorem \ref{extension}, if ${\Ker}(\rho)$ does not contain
 any non zero Hopf $*$-ideal, the kernel of the $*$-representation
 \begin{align*}
 \theta : H &\longrightarrow H//A \times \mathcal B(H\underline{\otimes}_AV) \\
 x &\longmapsto (p(x), \tilde{\rho}(x))
\end{align*}
does not contain any non-zero Hopf $*$-ideal, and hence $H$ is inner unitary.
\end{proof}

\begin{corollary}
 The compact Hopf $*$-algebra $A_o^*(n)$ is inner unitary.
\end{corollary}

\end{document}